\documentclass[12pt]{amsart}

\usepackage{geometry}

\numberwithin{equation}{section} 
\numberwithin{figure}{section} 
\theoremstyle{plain}
\newtheorem{thm}{Theorem}[section]
\theoremstyle{plain}
\newtheorem{lemma}[thm]{Lemma}
\theoremstyle{plain}
\newtheorem{prop}[thm]{Proposition}
\theoremstyle{plain}
\newtheorem{cor}[thm]{Corollary}
\theoremstyle{remark}
\newtheorem{remark}[thm]{Remark}
\theoremstyle{remark}
\newtheorem{note}[thm]{Notation}
\theoremstyle{remark}

\theoremstyle{plain}
\newtheorem*{prop*}{Proposition}

\theoremstyle{remark}
\newtheorem*{rem*}{Remark}

\theoremstyle{plain}

\theoremstyle{remark}

\theoremstyle{definition}

\theoremstyle{definition}
\newtheorem{definition}[thm]{Definition}

\newtheorem*{term}{Terminology}

\usepackage{latexsym,amsfonts, xypic,amssymb}
\xyoption{all}

\newcommand{\C}{\mathbb{C}}

\newcommand{\N}{\mathbb{N}}

\newcommand{\ba}{\mathbb{A}}

\newcommand{\ga}{\mathbb{G}_a}

\newcommand{\der}{\partial}

\newcommand{\LND}{\textup{LND}}
\newcommand{\ML}{\textup{ML}}
\newcommand{\Aut}{\textup{Aut}}

\newcommand{\id}{\textup{id}}

\newcommand{\SAut}{\textup{SAut}}

\def\A{\mathcal A}

\begin{document}

\title[Automorphism groups of Koras-Russell threefolds]{Automorphism groups of Koras-Russell threefolds of the first kind}

\author{L. Moser-Jauslin}
\address{Universite de Bourgogne ,
Institut de Math\'ematiques de Bourgogne CNRS -- UMR 5584,
9~avenue Alain Savary - B.P. 47 870, 21078 Dijon, France}
\email{lucy.moser-jauslin@u-bourgogne.fr }
 \thanks{The author is a member of the ANR program COMPLEXE (ANR 08-JCJC- 0130-01)}
\ 
\def\A{{\mathcal A}}

\begin{abstract}
 Koras-Russell threefolds are certain smooth contractible complex hypersurfaces in $\ba^4$ which are not algebraically isomorphic to $\ba^3$. One of the important examples is the cubic Russell threefold, defined by the equation  $x^2y+z^2+t^3+x=0$. In \cite{d-mj-p}, the automorphism group of the Russell cubic threefold was studied. It was shown, in particular, that all automorphisms of this hypersurface extend to automorphisms of the ambient space. This has several interesting consequences, including the fact that one can find  another hypersurface which is isomorphic to the Russell cubic, but such that the two hypersurfaces are inequivalent. In the present article, we will discuss how some of these results can be generalized to the class of Koras-Russell threefolds of the first kind.\end{abstract}

\maketitle

\section{Introduction}

In this article, we will study the automorphism groups of certain Koras-Russell threefolds. These varieties were first introduced by Koras and Russell when they were proving their remarkable result that all algebraic actions of $\C^*$ on affine three-space are linearizable \cite{K-R} (see also \cite{K-K-ML-R}). In order to prove this, they studied hyperbolic $\C^*$-actions on more general smooth contractible threefolds. This led them to 
introduce a set of threefolds which are smooth affine and contractible, however not isomorphic to $\ba^3$. These varieties are known as Koras-Russell threefolds.  One of the families of these varieties is given by hypersurfaces in $\ba^4_\C=\ba^4$ of the following form.  Let $d\ge 2$, and let $2\le k<l$ with $k$ and $l$ relatively prime. We denote by $X_{d,k,l}$ the zero set of $P_{d,k,l}={x^dy+z^k+t^l+x}$. That is, we have that $X_{d,k,l}=V(P_{d,k,l})$, the zero set of $P_{d,k,l}$.  We will call these Koras-Russell threefolds of the first kind. The case where $d=k=2$ and $l=3$ is known as the Russell cubic. In order to prove the linearizability of $\C^*$-actions on $\ba^3$, it was necessary to show that none of the Koras-Russell threefolds are isomorphic to $\ba^3$. For some of these varieties, this could be done with geometric invariants. However, for the Russell cubic, this problem was quite difficult. It was finally shown to be distinct from $\ba^3$ using the Makar-Limanov invariant \cite{ML}. This  tool has since become a very important and useful tool to study affine algebraic varieties.

It was shown by Makar-Limanov and Kaliman \cite{K-ML1997} that all of the Koras-Russell threefolds have non-trivial Makar-Limanov invariant, and are therefore not isomorphic to $\ba^3$. The  varieties of the first family have Makar-Limanov invariant isomorphic to a polynomial ring in one variable. These are the varieties we look at in this article.

Many authors, in studying exotic structures, are particularly interested in varieties with trivial Makar-Limanov invariant. However, here, and in several previous works, we take just the opposite approach. The fact that the Makar-Limanov invariant is not trivial gives strong restrictions on the automorphism group, which allows us to prove several interesting properties of these varieties. 

Consider  the variety $X_{2,2,3}$. As stated before, it is  known as the Russell cubic threefold.
In \cite{d-mj-p}, in collaboration with A. Dubouloz and P.M. Poloni, we studied the automorphism group of the Russell cubic. (In that article, it was called the Koras-Russell cubic threefold.) We showed in particular, that when the Russell cubic is embedded into $\ba^4$ as a hypersurface defined by $P=x^2y+z^2+t^3+x$, then all automorphisms extend to automorphisms of $\ba^4$. This leads to several interesting consequences. First of all, we gave an example of a polynomial $Q$ of four variables whose zero set is isomorphic to the Russell cubic, but such that not all automorphisms extend to automorphisms of $\ba^4$. In other words, we find inequivalent embeddings of the Russell threefolds in $\ba^4$. Also, we showed that the Russell threefold has a particular point (the origin of $\ba^4$) which is fixed by all automorphisms. 

Our goal here is to show that most of these results hold for all Koras-Russell threefolds of the first kind. Some of the arguments from \cite{d-mj-p} carry over with few changes, however, in order to study the automorphism group in the generalized setting, we need another result concerning the construction of certain automorphisms (see lemma \ref{hamilton}). 

In this article, we will review some of the  main ideas of the previous results, and we will show how they can be generalized for this class of Koras-Russell threefolds.

{\bf Acknowledgements :}{ I would like to thank David Wright for helpful conversations concerning  lemma {\ref{hamilton}}.

\section{The Makar-Limanov invariant and the Derksen invariant}

In this section, we will  give the definition of the Makar Limanov invariant for an affine variety, and describe the result which is the starting point of our study.

 Let $X$ be an affine irreducible complex variety with coordinate ring $\C[X]$. A locally nilpotent derivation $\der$ on $\C[X]$ is a $\C$-derivation of the algebra $\C[X]$ such that for any $f\in\C[X]$, there exists $n\in\N$ with $\der^n(f)=0$. We write $\der\in \LND(\C[X])$. It is well-known that there is a bijective correspondence between LND's on $\C[X]$ and algebraic actions of the additive complex group $\ga$ on $\C[X]$. 
The action corresponding to $\der$ is given by the map $\exp(t\der)$ from $\C[X]$ to $\C[X]\otimes\C[t]$. This, in turn, corresponds to an algebraic $\ga$-action on $X$. An element $P\in\C[X]$ is in the kernel of an LND, if and only if the hypersurfaces of $X$ defined $P=$ a constant are all stable by the corresponding action.  

\begin{definition}The {\it Makar-Limanov invariant} $\ML(\C[X])$ is the subring of $\C[X]$ defined as the intersection of the kernels of all locally nilpotent derivations on $\C[X]$.  One says that the Makar-Limanov invariant is {\it trivial} if it is equal to $\C$. The {\it Derksen invariant} is the subring of $\C[X]$ generated by all the kernels of LND's on $\C[X]$. One says that the Derksen invariant is {\it trivial} if it is equal to the ring $\C[X]$.
\end{definition}

It is clear, for example, that $\C[x,y,z]$ has a trivial Makar-Limanov invariant and trivial Derksen invariant, because  the intersection of the kernels of the three locally nilpotent derivations $\der_x=
\der/\der x$, $\der_y=
\der/\der y$ and $\der_z=
\der/\der z$ is already $\C$, and the ring generated by these three kernels is $\C[x,y,z]$. 

We define now the Koras-Russell threefolds of the first type. Given $d\ge 2$, and  $2\le k<l$ with $k$ and $l$ relatively prime, we denote by $X_{d,k,l}$ the zero set of $P_{d,k,l}={x^dy+z^k+t^l+x}$. That is, we have that $X_{d,k,l}=V(P_{d,k,l})$.

For any such variety, one can consider $\C[X_{d,k,l}]$ as a subring of $\C[x,z,t,x^{-1}]$, where $y=-(z^k+t^l+x)/x^d$. 

It was shown in \cite{K-ML1997} that the Makar-Limanov invariant is given by $\C[x]$. In fact, one can show also that   the Derksen invariant is given by $\C[x,z,t]$  (see \cite{K-ML}, the proof of theorem 9.1 and example 9.1).  Any automorphism of $X_{d,k,l}$ must preserve the two invariants. This allows us to describe the automorphism groups in terms of certain subgroups of the group of automorphisms of $\C[x,z,t]$. This will be described in detail in  section \ref{sec:aut}.

\section {The special automorphism group of $\C[x]/(x^d)[z,t]$}

Before studying the automorphism group of the Koras-Russell threefolds, we look at some properties of a related group of automorphisms.

Let $R=\C[x]$. Consider the Poisson bracket on $R[z,t]$ given by $\{f,g\}=f_zg_t-g_zf_t$. (Here, $f_z=\der f/\der z$ and $f_t=\der f/\der t$ for any $f\in R[z,t]$.) The ring $R[z,t]$ endowed with the Poisson bracket is a Lie algebra, and, in particular, the Poisson bracket satisfies the Jacobian identity.

For any $H\in R[z,t]$, we can define a derivation $D_H$ on $R[z,t]$ by defining $D_H(f)=\{H,f\}=H_zf_t-H_tf_z$.  Then the map $\widehat{\varphi_{wH}}=\exp(wD_H)$ is a well defined automorphism of $R[z,t][[w]]$ which fixes all elements of $R$ and fixes $w$. Its inverse is $\widehat{\varphi_{-wH}}$.  Note that $wD_H$ is a derivation, but not, in general, locally nilpotent. Therefore $\widehat{\varphi_{wH}}$ does not give an automorphism of the polynomial ring $\C[x,z,t,w]$. However, over the formal power series in $w$, we do get an automorphism. The determinant of  the Jacobian is therefore an invertible function in $R[z,t][[w]]$. We can even say more.

\begin{lemma}\label{hamilton} For any $H\in\C[x,z,t]$. Then the determinant of the Jacobian of $\widehat{\varphi_{wH}}$ is 1. 
\end{lemma}

\begin{remark} The automorphism $\widehat{\varphi_H}$ is a formal version of a Hamiltonian flow for the Hamiltonian $H$, and the lemma simply expresses the fact that Hamiltonian flows preserve volume.

This result has been proven in a more general setting in \cite{Now-1}, Corollary 2.4.9 and \cite{Now-2}, Corollary 2.4. We will give the main argument here. 
\end{remark}
\begin{proof} To simplify the notation, we denote $\varphi=\widehat{\varphi_{wH}}$, and $D=D_H$. We have that the determinant of the Jacobian is given by $det(Jac)=\{\varphi(z),\varphi(t)\}$. Now we calculate this value :

\begin{align*} 
det(Jac)&=\sum_{l=0}^\infty\frac{ w^l}{l!}D^l(z)_z\sum_{m=0}^\infty \frac{w^m}{m!}D^m(t)_t-\sum_{l=0}^\infty \frac{w^l}{l!}D^l(z)_t\sum_{m=0}^\infty \frac{w^m}{m!}D^m(t)_z\\
&=\sum_{n=0}^\infty \frac{w^n}{n!}\sum_{k=0}^n\binom{n}{k}((D^kz)_z(D^{n-k}t)_t-(D^kz)_t(D^{n-k}t)_z)\\
&=\sum_{n=0}^\infty \frac{w^n}{n!}\sum_{k=0}^n\binom{n}{k}\{(D^kz),(D^{n-k}t)\})\\
\end{align*}

Now let $\Sigma_n=\sum_{k=0}^n\binom{n}{k}\{(D^kz),(D^{n-k}t)\})$. By using the Jacobian identity and the definition of the derivation $D$, one finds that $D(\Sigma_n)=\Sigma_{n+1}$. Also, $\Sigma_0=1$. We therefore find by induction that $\Sigma_n=0$ if $n\ge 1$, and therefore $det(Jac)=1$.
 
\end{proof}

Throughout this article, we fix an integer $d\ge 2$. Denote by $\overline{R}$ the quotient ring $R/(x^d)$.   We call an automorphism of $\overline{R}[z,t]$ which fixes $\overline{R}$ {\it special} if the determinant of its Jacobian is a non-zero constant.  We call $\SAut_{\overline R}\overline R[z,t]$ the group of special automorphisms. We will use  lemma \ref{hamilton} to determine a set of generators of the subgroup of special  automorphisms of $\overline R[z,t]$ which are congruent to the identity modulo $(x)$.

For any $H\in \C[x,z,t]$, $\widehat{\varphi_{wH}}$ induces an automorphism $\overline{\varphi}_{xH}^d=\overline{\varphi}_{xH}$ of 
 $\overline R[z,t]\cong \C[x,z,t][[w]]/(w-x,w^d)$. In other words, $\overline{
\varphi}_{xH}$ is the exponential of the locally nilpotent derivation $xD_H$ of $\overline{R}[z,t]$. By the lemma, we have that $\overline{\varphi}_{xH}$ is special with Jacobian determinant equal to 1.

For $j=1,\ldots,d$, we define :
\begin{align*}
&{\A}_j=\{ \varphi\in \Aut_{ R} R[z,t] \mid \varphi\equiv \id \mod (x^j)\}\\
&\overline{\A}^d_j=\overline{\A}_j=\{ \overline\varphi\in \SAut_{\overline R}\overline R[z,t] \mid \overline\varphi\equiv \id \mod (x^j)\}.
\end{align*}

\begin{note} Throughout this article, the integer $d$ remains fixed. When we consider automorphisms  of $\A_j$ modulo $(x^d)$, we simplify the notation of $\overline\A_j^d$ to $\overline\A_j$. There will be one proof (for corollary \ref{fixed-point} ) where will consider $\overline\A_j^d$ and $\overline\A_j^{d+1}$, and we will need to distinguish between the two different cut-off points of the series in $x$. Thus, for that case, we will specify the two different notations.
\end{note}

It is clear that $\overline{\A}_d=\{id\}$, and that $\overline{\A}_{j+1}$ is a normal subgroup of $\overline{\A}_j$. 

  It was shown in \cite{E-M-V} that the natural map $\pi=\pi_d: \Aut_R[z,t]\to \SAut_{\overline R}\overline R[z,t]$ given by restriction modulo $(x^d)$ is surjective.
 In other words, any special automorphism modulo $(x^d)$ lifts to an automorphism of the polynomial ring $R[z,t]$. 
 From this result, one sees that the group $\A_1/\A_d\cong \overline{\A}_1$ and that for $j=1,\ldots, d-1$, we have that $\A_j/\A_{j+1}\cong\overline\A_j/\overline\A_{j+1}$. The automorphisms of $\A_1$ have Jacobian determinant equal to 1.

\begin{prop}\label{aut3space} For $j=1,\ldots,d-1$, we have that $\overline{\A}_j/\overline{\A}_{j+1}\cong {\A}_j/{\A}_{j+1}$ is congruent to the additive group of the ideal $(z,t)\in\C[z,t]$.
\end{prop}

\begin{proof}
For $j=1,\ldots, d-1$, we will show that $\overline{\A}_j/\overline{\A}_{j+1}$ is congruent to the additive group given by the 
ideal $(z,t)\subset \C[z,t]$. To do this, we construct a surjective group homomorphism $\Phi_j :\overline{\A}_j\to (z,t)\subset\C[z,t]$ whose kernel is $\overline{\A}_{j+1}$. If $\overline\varphi\in\overline{\A}_j$ then, by using the condition that the jacobian determinant of $\overline{\varphi}$ is one, it is easy to see that there exists $h\in (z,t)\subset\C[z,t]$ such that $\overline\varphi(z)\equiv z+ h_tx^j \mod 
(x^{j+1})$ and $\overline\varphi(t)\equiv t-h_tx^j\mod ( x^{j+1})$. We define the map  $\Phi_j$ from $\overline{\A}_j/\overline{\A}_{j
+1}$ to $(z,t)$ by $\Phi_j(\overline\varphi)=h$. We now find a preimage for any $h\in(z,t)$ to show that $\Phi_j$ is surjective. We consider   $\overline\varphi_{x^jh}=\exp(x^jD_h)$ the corresponding automorphism of $\overline R[z,t]$. We have that $\Phi_j(\overline\varphi_{x^jh})=h$, and by 
lemma \ref{hamilton}, the Jacobian  determinant is 1, and therefore, we have found a preimage of $h$ in $\overline{\A}_j$.

\end{proof}

In particular, we have shown the following. For any $h\in (z,t)$ and any $j=1,\ldots, d-1$, choose any automorphism $\varphi_{x^jh}$ which restricts to $\overline\varphi_{x^jh}$ modulo $(x^d)$. (The choice is not unique.) 

\begin{cor}  Let $\A_1$ and $\overline\A_1$ be defined as above. Then :
\begin{itemize}
\item $\overline{\A}_1$ is generated by  the set $\{ \overline\varphi_{x^jh} \mid j\in\{1,\ldots,d-1\} ; h\in (z,t)\}$; and 
\item ${\A}_1$ is generated by ${\A_d}$ and the set $\{ \varphi_{x^jh} \mid j\in\{1,\ldots,d-1\} ; h\in (z,t)\}$. 
\end{itemize}
\end{cor}

 Instead of $\varphi_{x^jh}$ one could also choose any other automorphism which is equivalent to it modulo $(x^{j+1})$.

\section{The automorphism group of the Koras-Russell threefolds}\label{sec:aut}

Let $X=X_{d,k,l}$ be one of the Koras-Russell threefolds defined above. We will study the automorphism group $\A$ of $\C[X]$. 

As was done in \cite{d-mj-p}, using the Makar-Limanov invariant and the Derksen invariant, we see that any automorphism $\tilde\varphi$ of $\C[X]$ restricts to an automorphism  $\varphi$ of $ \C[x,z,t]$ which stabilizes the ideal $(x)$ and the ideal $I=(x^d,z^k+t^l+x)$. Also, any automorphism of this form of $\C[x,z,t]$ lifts to a unique automorphism of $\C[X]$. Thus the automorphism group of $\C[X]$ is isomorphic to the subgroup of the automorphism group of $\C[x,z,t]$ of those automorphims that stabilize $(x)$ and $I$. Using the $\C^*$-action on $X$, we can show that $\A$  is isomorphic to the semi-direct product $\A_1(d;z^k+t^l+x)\rtimes\C^*$, where $\A_1(d;z^k+t^l+x)$ is the subgroup of $R$-automorphisms of $R[z,t]$  which are congruent to the identity modulo $(x)$,  and which stabilize $I$. Thus, to study the automorphism group, we must determine $\A_1(d;z^2+t^3+x)$. 
\[ \A_1(d;z^k+t^l+x)=\{\varphi\in \Aut_{R}R[z,t] \mid \varphi(I)=I\}
\]

\begin{term} If $\tilde\varphi$ is an automorphism of $\C[X]$ which restricts to an automorphism $\varphi$ of $\C[x,z,t]$, we say that {\it $\varphi$ lifts to the automorphism $\tilde\varphi$}. From what we have seen, an automorphism lifts if and only if $(x)$ and $I$ are preserved. Also, the lifting is unique. Throughout this article, given an automorphism $\varphi$ of $\C[x,z,t]$ which preserves the ideals $(x)$ and $I$, we will denote by $\tilde\varphi$ the unique lift to an automorphism of $\C[X]$. More precisely, if $\varphi(z^k+t^l+x)=a(z^k+t^l+x)+bx^d$ with $a, b\in\C[x,z,t]$, then we have that $\tilde\varphi(y)=ay-b$. The inverse ${\tilde\varphi}^{-1}$ is given by the lift $\widetilde{\varphi^{-1}}$ of the inverse of $\varphi$.
\end{term}

First, we consider a more general setting.
Throughout this section, we fix $r\in\C[x,z,t]$  such that the zero set of $r_0=r(0,z,t)$ is a connected curve in $\ba^2$, and $r_0$ has no multiple factors. For example, we can choose $r=z^k+t^l+x$ where $k$ and $l$ are relatively prime. 
We study certain subgroups of the automorphism group of affine three space.
Let $I_{d,r}$ be the ideal generated by $r$ and $x^d$.  Now we define the following subgroups of the automorphism group of $\C[x,z,t]$. 

\begin{align*}
\A_j(d;r)&=\{ \varphi\in \A_j \mid \varphi(I_{d,r})=I_{d,r}\}\\
\overline\A_j(r)&=\{ \varphi\in \overline\A_j \mid \varphi(r)\in(r)\}
\end{align*}

Note that $\A_d(d;r)=\A_d$. Also, note that $\pi^{-1(}\overline\A_j(r))=\A_j(d;r)$. From this, it is easy to see that $\A_1(d;r)/\A_d\cong\overline\A_1(r)$, and that  $\A_j(d;r)/\A_{j+1}(d;r)\cong \overline\A_j(r)\overline/\A_{j+1}(r)$ for $j=1,\ldots,d-1$.

\begin{prop} For $j=1,\ldots, d-1$, the group $$\A_j(d;r)/\A_{j+1}(d;r)\cong\overline \A_j(r)/\overline\A_{j+1}(r)$$ is isomorphic to the additive group $\C[z,t]$.
\end{prop}

\begin{proof} We use the lemma \ref{hamilton} together with the idea of the proof of the first case, for the Russell cubic  given in \cite{d-mj-p}, to  construct a surjective homomorphism $\Psi_j$ from $\overline\A_{j}(r)$ to $\C[z,t]$ whose kernel is  $\overline\A_{j+1}(r)$. 

Let $\overline\varphi$ be an automorphism in $\overline\A_{j}(r)$. By calculating the Jacobian of the automorphism as before, we have that there exists $h_0\in\C[z,t]$ such that 
\begin{align*}
\overline\varphi(z)&\equiv z+x^j(h_0)_t \mod ( x^{j+1}) \\
\overline\varphi(t)&\equiv t-x^j(h_0)_z \mod ( x^{j+1}). \\
\end{align*}

Now $\overline\varphi(r)\equiv r+x^j\{r_0,h_0\} \mod ( x^{j+1})$, and since the ideal $(r)$ of $R[z,t]$ is stable, and $j-1< d$, we have that $\{r_0,h_0\}$ is in the ideal $(r_0)$ of $\C[z,t]$ generated by $r_0$. Now, using the fact that $r_0$ has no multiple factors and that the zero set of $r_0$ in $\ba^2$ is connected, this implies that, up to a constant, $h_0\in (r_0)$. This comes from the fact that, by the hypotheses on $r_0$, the restriction to the zero set of $r_0$ of the function defined by $h_0$ is  constant. We may therefore  assume that $h_0$ is in the ideal generated by $r_0$.  (see Proposition 3.6 of  \cite {d-mj-p}). If $h_0=\alpha r_0$, we define $\Psi_j(\overline\varphi)=\alpha$. This gives a homomorphism from $\overline\A_j(r)$ to the additive group $\C[z,t]$, whose kernel is exactly $\overline\A_{j+1}(r)$. Now we are left to show that $\Psi_j$ is surjective. For any $\alpha\in\C[z,t]$, we find a preimage in $\overline\A_j(r)$. Let $h=\alpha r$, and let $\overline\varphi$ be equal to $\overline\varphi_{x^jh}$. We have that $\overline\varphi\in\overline\A_j$, and since $r$ is a factor of $h$, we have that the ideal $\overline\varphi(r)\in(r)\subset R[z,t]$. Also, since $r\equiv r_0 \mod ( x)$, we have that $\Psi_j(\overline\varphi)=\alpha$. 

\end{proof}
As before, for any $h\in R[z,t]$ and any $j\in\{1,\ldots,d-1\}$, we  choose an arbitrary preimage $\varphi_{x^jh}\in\pi^{-1}(\overline\varphi_{x^jh})\subset \A_1$. Note that, since  $\pi^{-1}(\overline\A_j(r))=\A_j(d;r)$, we have that $\varphi_{x^jh}\in\A_j(d;r)$.

\begin{remark} Note the difference between the proof of this proposition and the proof of proposition \ref{aut3space}.
In the previous section, we  could choose generators of $\A_j/\A_{j+1}$ of the form $\varphi_{x^jh}$ where $h$ is independant of $x$.  Now, we need to find elements of $\A_j$ which preserve the ideal $(r)$. In order to do this, we choose $h$ in the ideal $(r)$. In particular, $h$ depends on $x$. If $\alpha\in\C[z,t]$, let $h=h(x,z,t)=\alpha r$  and $h_0=h_0(z,t)=\alpha r_0-h(0,0,0)$.  Then $h_0\in (z,t)\in\C[z,t]$, and $\varphi_{x^jh}$ is congruent to $\varphi_{x^jh_0}$ modulo $\A_{j+1}$.
\end{remark}

As an immediate consequence of the proposition, we have the following set of generators of $\A_1(d;r)$ :

\begin{cor} \label{generators} Suppose that the zero set of $r_0$ is connected, and that $r_0$ has no multiple factors. Then
\begin{itemize}
\item  the group $\overline\A_1(r)$ is generated by  the set of automorphisms of the form $\overline\varphi_{x\gamma r}$, for $\gamma\in \overline R[z,t]$;
\item  the group $\A_1(d;r)$ is generated by $\A_d$ and the set of  automorphisms of the form $\varphi_{x\gamma r}$, for $\gamma\in R[z,t]$.
\end{itemize}
\end{cor}

\section{Extension of automorphisms to $\ba^4$}

\begin{term} Suppose $\tilde\varphi$ is an automorphism of $\C[X]$. We say that {\it $\Phi$ defines an extension of $\tilde\varphi$} if $\Phi$ is an automorphism of $\C[x,y,z,t]$ which preserves the ideal generated by $P=P_{d,k,l}$, and $\tilde\varphi$ is the induced automorphism of $\C[X]$. In geometric terms, this simply means that the automorphism of $X$ corrsesponding to $\tilde\varphi$ extends to an automorphism of $\ba^4$. This extension is not necessarily unique.
\end{term}

The hypersurfaces $X_{d,k,l}$ are embedded in affine four-space. In this section, we will show that all automorphisms of $X_{d,k,l}$ extend to automorphisms of the ambient four-space. This section uses a technique, developed in \cite{poloni}, used also in \cite{mj-p},  and was adapted for the case of the Russell cubic in \cite{d-mj-p}. The idea is simply to lift families of automorphisms. More precisely, we use the fact that we can construct automorphisms of all the fibers of the polynomial $P_{d,k,l}$ in a continuous way.

\begin{thm} All automorphisms of $X_{d,k,l}$ extend to automorphisms of $\ba^4$.
\end{thm}

\begin{proof}

We set $r=z^k+t^l+x$. First of all, the action of $\C^*$ extends to $\ba^4$, and thus we only must check the automorphisms lifted from elements in $\A_1(d;r)$. 

Now, if $\varphi\in\A_d$, then we can easily see that the automorphism $\tilde\varphi$  of $\C[X]$ extends to an automorphism of $\C[x,y,z,t]$. Indeed, if $\varphi(r)=r+bx^d$ with $b\in\C[x,z,t]$, then   $\tilde\varphi(y)=y-b$. We now give an automorphism $\Phi$ of $\C[x,y,z,t]$ which defines an extension of $\tilde\varphi$. We choose $\Phi|_{\C[x,z,t]}=\varphi$, and $\Phi(y)=y-b$.  This is an automorphism, because the image of $y$ differs from $y$ by a polynomial in the other variables.

Now, we use that $\A_1(d;r)$ is generated by $\A_d$ and automorphisms of the form $\varphi_{x\gamma r}$ where $\gamma$ is in $\C[x,z,t]$.  We must show that we can pick an extension for any  automorphism of the form $\tilde\varphi=\tilde\varphi_{x\gamma r}$ where $\varphi=\varphi_{x\gamma r}$. This is not at all obvious, because if we try to use a similar argument to the one above, we find that $\varphi(r)=ar+bx^d$, where $b\in\C[x,z,t]$ and $a\in\C[x,z,t]$. Then if we try to define $\Phi|_{\C[x,z,t]}=\varphi$, and $\Phi(y)=ay-b$, we have an endormorphism of $\C[x,y,z,t]$ which is not an  automorphism, since $a$ is not invertible in $\C[x,z,t]$.  We therefore must adapt this endomorphism by adding appropriate elements of the ideal $(P)\in\C[x,y,z,t]$ to construct an automorphism. This is done exactly as in the proof of  Lemma 4.2 of \cite{d-mj-p}.  We describe the process here.

Let $r_\lambda=r-\lambda$ for any $\lambda\in\C$. We will construct a family of automorphisms $\varphi_\lambda$ such that the following properties hold. 
\begin{enumerate}
\item[(1)] For each $\lambda\in\C$, $\varphi_\lambda\in\A_1(d;r-\lambda)$;
\item[(2)] For $\lambda=0$, $\varphi_0$ is congruent to $\varphi$ modulo $\A_d$; and
\item[(3)] The map $\C\to\A_1$ given by $\lambda\mapsto\varphi_\lambda$ is polynomial in $\lambda$.
\end{enumerate}

 Let $H=\gamma(r-\lambda)\in\C[\lambda,x,z,t]$, and let $D_{H}$  be the $\overline{R}[\lambda]$-derivation on $\overline{R}[\lambda][z,t]$ given by $D_{H}(g)=H_zg_t-H_tg_z$. We have that $xD_H$ is locally nilpotent as a derivation on $\overline{R}[\lambda][z,t]$. Let $\overline{F}=\exp({xD_H})$. It is a special automorphism in $\Aut_{\overline{R}[\lambda]}\overline{R}[\lambda][z,t]$, by lemma \ref{hamilton}, and, therefore, by the result of \cite{E-M-V}, there exists and automorphism $F\in \Aut_{R[\lambda]}R[\lambda,z,t]$ which restricts to $\overline F$ modulo $(x^d)$. For any $\lambda\in\C$, let $\varphi_\lambda(g)=F(g)\in R[\lambda,z,t]$ for any $g\in R[z,t]$. By construction, these automorphisms satisfy the three properties above.

 Now we construct $\Phi$. First, we show that $F$ extends to an automorphism $\tilde{F}$  of the ring
 $R[\lambda][y,z,t]/(P-\lambda)$, which contains $R[\lambda][z,t]$ as a subring. Note that, by construction, $\overline{F}(r-\lambda)$ belongs to the ideal $(r-\lambda)$, and therefore, $F(r-\lambda)$ is in the ideal $(r-\lambda,x^d)\subset R[\lambda][z,t]$. Suppose that $F(r-\lambda)=(r-\lambda)\alpha+ x^d\beta$ with $\alpha$ and $\beta$ elements of $R[\lambda][z,t]$. Then we extend the automorphism $F$ to an endomorphism $\tilde{F}$ of $R[\lambda][z,t,y]/(P-\lambda)$ by posing $\tilde{F}|_{R[\lambda][z,t]}=F$ and $\tilde{F}(y)=y\alpha -\beta$. (One can check  easily that  this makes sense, since $\tilde{F}(P-\lambda)=(P-\lambda)\alpha=0\in R[\lambda](z,t,y]/(P-\lambda)$.) Note that $\tilde{F}$ is the unique lifting of the automorphism  $F$  of $R[\lambda][z,t]$ to an endomorphism of $R[\lambda][y,z,t]/(P-\lambda)$.  Now we remark that $\tilde F$ is in fact an automorphism. To see this, note that if $G$ is the inverse of $F$, then $\tilde G$  is the inverse of $\tilde F$. 
 
Now the projection $R[\lambda,y,z,t]\to R[y,z,t]$ which sends $\lambda$ to $P$ induces an isomorphism between $R[\lambda,y,z,t]/(P-\lambda)$ and $R[y,z,t]=\C[x,y,z,t]$. Thus $\tilde{F}$ induces an automorphism $\Phi$ of $\C[x,y,z,t]$. By the construction above, $\Phi(P)=(P)$, and for any $\lambda\in\C$, we have that $\Phi$ induces the automorphism $\tilde\varphi_\lambda$ on $\C[x,y,z,t]/(P-\lambda)$.
 
 (Another way to see the construction is that we have simply replaced $\lambda$ by $P$ in the automorphism $\tilde F$. Thus, in the geometric setting, $\tilde\varphi_\lambda$ induces an automorphism $\phi_\lambda$ of $V(P-\lambda)\subset\ba^4$ for any $\lambda$, and the  map of $\ba^4$ induced by $\Phi$ restricts to  $\phi_\lambda$ on $V(P-\lambda)$ for any $\lambda$.)

\end{proof}

\section{The fixed point and some comments on inequivalent hypersurfaces}

In \cite{d-mj-p}, we showed that all automorphims of the Russell cubic threefold have a common fixed point, namely, the origin of $\ba^4$. We will now show that this result also holds for all Koras-Russell threefolds of the first kind. The proof is not the same as in \cite{d-mj-p}.

\begin{cor}\label{fixed-point} Any automorphism of $X=X_{d,k,l}$ fixes the origin.
\end{cor}
\begin{proof} The $\C^*$-action on $X$ fixes the origin, so we now consider automorphisms lifted from $\A_1(d;r)$, where $r=z^k+t^l+x$. Note that the ideal $(x,z,t)$ is stable by any automorphism $\varphi$ of $\A_1(d;r)$, since the origin is the unique singular point of the zero set of the ideal $(x,r)$ in $\C[x,z,t]$.   We know that $\varphi$ lifts to an automorphism $\tilde\varphi$ of $\C[X]$ determined as follows. If $\varphi(r)=ar+bx^d$ with $a,b\in \C[x,z,t]$, then $\tilde\varphi(y)=ay-b$. We thus are left to check that $b\in (x,z,t)$, that is, has no constant term.
First consider any automorphism  $\varphi$ of  $\A_d$.  In this case, $\varphi$ is congruent to the identity modulo $(x^d)$, and thus we can choose $a=1$. We have that $\varphi(r)=r+bx^d$, where $b\in\C[x,z,t]$, and $\varphi(z)\equiv z+h_tx^d \mod ( x^{d+1})$ and $\varphi(t)\equiv t-h_zx^{d} \mod ( x^{d+1})$. Then we find that $b(0,z,t)=\{r_0,h\}\in ((r_0)_z, (r_0)_t)$. Since $r_0=z^k+t^l$ and $2\le k,l$,  we have that $b\in (x,z,t)$. 

Now we know that the group of automorphisms is generated by $\A_d$ and ones of the form $\varphi_{x\gamma r}$ where $\gamma\in\C[x,z,t]$.  Thus we are left to consider the automorphisms of the form $\varphi_{x\gamma r}$. Remember that $\varphi_{x\gamma r}$ was chosen as an arbitrary pre-image of $\overline\varphi_{x\gamma r}\in \overline\A_1$ by the homomorphism $\pi=\pi_d$. Thus, we may assume that $\varphi_{x\gamma r}$ is in $\A_1(d+1;r)$. In other words, we can choose $\varphi_{x\gamma r}$ as an element in the pre-image of $\overline\varphi_{x\gamma r}^{d+1}$ by the homomorphism $\pi_{d+1}:\A_1\to \overline\A_1^{d+1}$. This gives that $\varphi_{x\gamma r}(r)$ is in the ideal $(r,x^{d+1})$, and therefore, we may choose $b\in (x)$. 
\end{proof}

\begin{definition}Let $X$ and $Z$ be two affine varieties. An embedding $\varphi: X\to Z$ is  an injective morphism whose image is closed, and which induces an isomorphism between $X$ and $\varphi(X)$. Two embeddings $\varphi_j:X\to Z$  $(j=1,2)$  are said to be {\it equivalent} if there exists an automorphism
$\Phi$  of $Z$ such that 	$\Phi\circ\varphi_1 = \varphi_2$. 
\end{definition}

\begin{definition}Two subvarieties $X_1$ and $X_2$ of an affine variety $Z$ are said to be {\it equivalent} if there exists an automorphism
	$\Phi$  of $Z$ such that 	$\Phi(X_1)$ = $X_2$. If $X_1$ and $X_2$ are hypersurfaces of $Z=\ba^n$,  then we say they are {\it equivalent
hypersurfaces}.
\end{definition}

These two notions are related, but they are not identical. To be more precise, we consider what happens in the case of two isomorphic  subvarieties $X_1$ and $X_2$  of $Z$, where $\phi:X_1\to X_2$ is an isomorphism. Let  $i_j:X_j\to \ba^n$ be the inclusion maps. Then one can ask of the embeddings $i_1$ and $i_2\circ\phi$ of $X_1$ into $\ba^n$ are equivalent as embeddings. The answer is affirmative if and only if the isomorphism $\phi$ extends to an automorphism of $Z$. One can also ask if $X_1$ and $X_2$ are equivalent as subvarieties. Here, the answer is affirmative if and only if there exists an isomorphism between $X_1$ and $X_2$ which extends to an automorphism of the ambient space. In particular, if two isomorphic subvarieties are not equivalent, then there exist inequivalent embeddings of the given subvariety. Finally, if all automorphisms of $X_1$ extend to automorphisms of $Z$, then the two notions become the same : $X_1$ is equivalent as a subvariety of $Z$ to $X_2$ if and only if $i_1$ and $i_2\circ\phi$ are equivalent embeddings. 

Many examples of inequivalent hypersurfaces exist in the literature. See, for example, \cite{S-Y}, \cite{kaliman} and \cite{f-mj} . However, it is a long-standing open problem to know if there is a hypersurface of $\ba^n$ isomorphic to $\ba^{n-1}$ but not equivalent to a hyperplane. For this reason, it is interesting to find inequivalent hypersurfaces  isomorphic to varieties that are, in some sense, close to affine space. In \cite{K-Z}, for example,  Kaliman and Zaidenberg gave examples of acyclic surfaces which admit
non-equivalent embeddings in three-space.  The obstruction to equivalence is of a topological nature. One can also ask what happens for Koras-Russell threefolds. In \cite{d-mj-p}, we give an  example of two inequivalent hypersufaces isomorphic to the Russell cubic. The methods are purely algebraic, because, in particular, they become equivalent as analytic embeddings. This leads to the natural question if all Koras-Russell threefolds have inequivalent representations as hypersurfaces. The method used for the Russell cubic only works in special cases. 

One way to show that two hypersurfaces $V(P)$ and $V(Q)$ are inequivalent is to show that the fibers of $P$ and $Q$ are different. In other words, if $V(P)$ and $V(Q)$ are equivalent as hypersurfaces, then for all $c\in\C^*$, there exists $c'\in\C^*$ such that $V(P-c)\cong V(Q-c')$. So if one can find non-isomorphic fibers, the hypersurfaces are inequivalent. This is, however, not a necessary condition. In \cite{mj-p} and in \cite{poloni}, there are many examples of Danielewski hypersurfaces of $\ba^3$ for which all fibers are isomorphic, but the hypersurfaces are inequivalent. The methods used there for surfaces can be partially adapted to the case of Koras-Russell threefolds. This was how the Russell cubic was treated. For the general case, work is in progress.

\bibliographystyle{amsalpha}

\end{document}